\documentclass{amsart}
\usepackage{amsmath, amssymb,epic,graphicx,mathrsfs,enumerate}
\usepackage[all]{xy}
\usepackage{color}
\usepackage{comment}

\usepackage{amsthm}
\usepackage{amssymb}
\usepackage{latexsym}
\usepackage{longtable}
\usepackage{epsfig}
\usepackage{amsmath}
\usepackage{hhline}

 \DeclareMathOperator{\perm}{Sym}

 \DeclareMathOperator{\frat}{Frat}
\DeclareMathOperator{\ssl}{SL}

  \DeclareMathOperator{\diam}{diam}

\DeclareMathOperator{\End}{End} 
\DeclareMathOperator{\dist}{dist}

\renewcommand{\emptyset}{\varnothing}

\newtheorem{thm}{Theorem}%[section]

 \newtheorem{lemma}[thm]{Lemma}
\newtheorem{prop}[thm]{Proposition}

\numberwithin{equation}{section}

\renewcommand{\footnote}{\endnote}
\newcommand{\ignore}[1]{}\makeglossary

\begin{document}
	\bibliographystyle{amsplain}
%	\subjclass{ 20D10, 05C25}
%	\keywords{groups generation; waiting time; Sylow subgroups; permutations groups}
	\title[Generating graph]{The generating graph of a profinite  group}

	\author{Andrea Lucchini}
	\address{Andrea Lucchini\\ Universit\`a degli Studi di Padova\\  Dipartimento di Matematica \lq\lq Tullio Levi-Civita\rq\rq\\ Via Trieste 63, 35121 Padova, Italy\\email: lucchini@math.unipd.it}
%	\thanks{Partially supported by Universit\`a di Padova (Progetto di Ricerca di Ateneo: \lq\lq Invariable generation of groups\rq\rq).}

	\begin{abstract} Let $G$ be 2-generated group. The generating graph $\Gamma(G)$ of $G$ is the graph whose
		vertices are the elements of $G$ and where two vertices $g$ and $h$ are adjacent if $G = \langle g, h \rangle.$ This definition can be extended to a 2-generated profinite group $G,$ considering in this case topological generation. We prove that the set $V(G)$ of non-isolated vertices of $\Gamma(G)$ is closed in $G$ and that, if $G$ is prosoluble, then the graph $\Delta(G)$ obtained from $\Gamma(G)$ by removing its isolated vertices is connected with diameter at most 3. However we construct an example of a 2-generated profinite group $G$ with the property that $\Delta(G)$ has $2^{\aleph_0}$ connected components. This implies that the so called \lq\lq swap conjecture\rq\rq \ does not hold for finitely generated profinite groups. We also prove that if an element of $V(G)$ has finite degree in the graph $\Gamma(G),$ then $G$ is finite. \end{abstract}
	\maketitle

%%%%%%%%%%%%%%% ABOUT DIAMETER %%%%%%%%%%%%%%%%%%%%%%%%%%%%%%%%%
\section{Introduction} Given a group $G,$ the generating graph $\Gamma(G)$ is the graph with vertex set $G$ where two elements $x$ and $y$ are adjacent if and only if $G=\langle x,y \rangle.$ There could be many isolated vertices in this graph. All of the elements in the
Frattini subgroup will be isolated vertices, but we can also find isolated vertices outside
the Frattini subgroup (for example, the  elements of the Klein subgroup are
isolated vertices in $\Gamma(\perm(4))).$ Several strong structural results about $\Gamma(G)$ are known in the case where $G$ is
simple, and this reflects the rich group theoretic structure of these groups. For example, if $G$ is a nonabelian simple group, then the only isolated vertex of $\Gamma(G)$ is the identity and the
graph $\Delta(G)$ obtained by removing the isolated vertex is connected with diameter two and, if $G$ is sufficiently large, admits a Hamiltonian cycle. In \cite{diam}, it is proved  that $\Delta(G)$ is connected, with diameter at most 3, if $G$ is a finite soluble group. 

\

Clearly the definitions of $\Gamma(G)$ and $\Delta(G)$ can be extended to the case of a {2-generated} profinite group $G$ (in this case we consider topological generation, i.e. we say that $X$ generates $G$ if the abstract subgroup generated by $X$ is dense in $G).$ We denote by $V(G)$ the set of the vertices of $\Delta(G)$. We will prove in section \ref{prelim} (see Proposition \ref{chiuso}) that $V(G)$ is a closed subset of $G.$ The profinite group
$G,$ being a compact topological group, can be seen as a probability space. If we denote
with $\mu(G)$ the normalized Haar measure on $G,$ so that $\mu(G) = 1,$ we may consider the probability $\mu(V(G))$ that a vertex of the generating graph $\Gamma(G)$ is non-isolated. By \cite[Remark 2.7(ii)]{scott}, if $G$ is a 2-generated pronilpotent group, then $\mu(V(G))\geq 6/\pi^2.$ However is it possible to construct a 2-generated prosoluble group $G$ with $\mu(V(G))=0.$ Indeed let $H = C_2^2$ and let $h_1$, $h_2$, $h_3$ be the nontrivial elements of $H$. For each odd prime $p,$  write $N_p = C_{p}^3$ and define
$
G = \left(\prod_{p \text { odd }}N_p \right) \rtimes H
$
where for each odd prime $p,$ the subgroup $N_p$ is $H$-stable and for $(n_{p,1},n_{p,2},n_{p,3}) \in N_p$ and $h_j \in H$
\[
(n_{p,1},n_{p,2},n_{p,3})^{h_j} = \left\{ \begin{array}{ll} (n_{p,1},n_{p,2}^{-1},n_{p,3}^{-1}) & \text{if $j=1$} \\ (n_{p,1}^{-1},n_{p,2},n_{p,3}^{-1}) & \text{if $j=2$} \\ (n_{p,1}^{-1},n_{p,2}^{-1},n_{p,3}) & \text{if $j=3$.} \end{array} \right.
\]  
The vertex $((n_{p,i})_{1\leq i\leq 3, p \text { odd} };h)$ is non-isolated in $\Gamma(G)$ if and only if $h=h_j$ for some $1 \leq j \leq 3$ and $n_{p,j} \neq 0$ for all $p$ (see \cite[Example 2.8]{scott} for more details). This implies
$$
\mu(V(G))=  \frac{3}{4} \prod_{p \text { odd}}\left(1-\frac{1}{p}\right)=0.
$$
The neighborhood  of a vertex $g$ of the generating graph of $G$, denoted by $\mathcal N_G(g),$ is
the set of vertices of $\Gamma(G)$ adjacent to $g$. The degree of $g$, denoted by $\delta_G(g),$ is the number of edges of $\Gamma(G)$ incident with $g.$ Since $G=\langle g, x\rangle$ if and only if there is no open maximal subgroup of $G$ containing $g$ and $x$, it follows  that
$$\mathcal N_G(g)=G\setminus \bigcup_{M\in \mathcal M_g} M,$$ where $\mathcal M_g$ is the set of the open maximal subgroups of $G$ containing $g.$ In particular $\mathcal N_G(g)$ is a closed subgroup of $V(G).$ The surprising result is that if  $\delta_G(g)$ is finite for some $g\in V(G),$ then $G$ is finite. Indeed we have:
\begin{thm}\label{infinito}
	Assume that $G$ is an infinite 2-generated profinite group. Then the degree of any non-isolated vertex in the generating graph $\Gamma(G)$ is infinite.
\end{thm}
This is a consequence of the following result, concerning the generating graph  of a finite group $G.$
\begin{thm}\label{tanti}
	Let $G$ be a 2-generated finite group. If $g\in V(G),$ then $\delta_G(g)\geq 2^{t-2},$ where $t$ is the length of a chief series of $G.$
\end{thm}
The results concerning the connectivity of $\Delta(G)$ when $G$ is a finite soluble group can be extended with standard arguments to the prosoluble case.

\begin{thm}\label{diquattro}
	If $G$ is a 2-generated finite soluble group, then  $\Delta(G)$ is connected and $\diam(\Delta(G)) \leq 3.$
\end{thm}

The bound $\diam(\Delta(G))\leq 3$ given in Theorem  \ref{diquattro} is best possible. In \cite{diam},  a soluble 2-generated group $G$ of order $2^{10}\cdot 3^2$ with $\diam(\Delta(G))=3$ is constructed.  However in \cite{diam}  it is proved that
$\diam(\Delta(G))\leq 2$ in some relevant cases. These results can be extended to the profinite case, with the same arguments used in the proof of  Theorem  \ref{diquattro}. Suppose that a 2-generated prosoluble soluble group $G$ has the property that $|\End_G(V)|>2$ for every nontrivial irreducible $G$-module $V$ which is $G$-isomorphic to a complemented chief factor of $G$.
	Then  $\diam(\Delta(G))\leq 2$.  In particular $\diam(\Delta(G))\leq 2$ if the derived subgroup of $G$ is pronilpotent or has odd order (as a supernatural number).
	
\

No example is known of a 2-generated finite group $G$ for which
$\Delta(G)$ is disconnected, and it is an open problem whether or not  $\Delta(G)$ is connected when $G$ is an arbitrary finite group.  The situation is different in the profinite case.

\begin{thm}\label{contro}
There exists a 2-generated profinite group $G$ with the property that $V(G)$ is the disjoint union of $2^{\aleph_0}$ connected components. Moreover each connected component is dense in $V(G).$
\end{thm}

The previous theorem implies that the \lq\lq swap conjecture\rq\rq \ is not satisfied by the 2-generated profinite groups. Recall that the swap conjecture concerns the connectivity of the graph $\Sigma_d(G)$ in which the vertices are the ordered generating $d$-tuples and  two vertices $(x_1,\dots,x_d)$ and $(y_1,\dots,y_d)$ are adjacent if and only if they differ only by one entry.
Tennant and Turner \cite{TT} conjectured that the swap graph is connected for every group. Roman'kov \cite{rom} proved that the free metabelian group of rank 3 does not satisfy this conjecture but no counterexample is known
in the class of finite groups. However, by Lemma \ref{swa} in section \ref{esempio},  it follows from Theorem \ref{contro} that there exists a 2-generated profinite group with the property that the graph $\Sigma_2(G)$ has 
$2^{\aleph_0}$ connected components.

\section{Some properties of $\Gamma(G)$}\label{prelim}

We begin this section, proving a criterion to decide when a vertex $x$ of the generating graph $\Gamma(G)$ of a profinite group $G$ is non-isolated.
\begin{lemma}Let $G$ be a 2-generated profinite group. Then
	$x\in V(G)$ if and only if $xN \in V(G/N)$ for every open normal subgroup $N$ of $G.$
\end{lemma}
\begin{proof}
Let $\mathcal N$ be the set of the open normal subgroups of $G$. Assume $xN \in V(G/N)$ for every $N\in \mathcal N.$ Given $N\in \mathcal N,$ let $\Omega_N=\{y\in G
\mid \langle x,y \rangle N=G\}.$ Notice that if $y\in \Omega_N,$ then $yN\subseteq \Omega_N$ and consequently $\Omega_N$, being a union of cosets of the open subgroup $N,$ is a non-empty closed subset of $G.$ If $N_1,\dots,N_t\in \mathcal N,$ then $\emptyset \neq \Omega_{N_1\cap \dots \cap N_t}\subseteq \Omega_{N_1}\cap \dots \cap \Omega_{N_t}.$ Since $G$ is compact, $\cap_{N\in \mathcal N}\Omega_N\neq \emptyset.$ Let $y\in \cap_{N\in \mathcal N}\Omega_N.$ Since $\langle x,y \rangle N=G$ for every $N\in \mathcal N,$ we have $\langle x,y \rangle=G,$ and consequently $y\in V(G).$
\end{proof}

\begin{prop}\label{chiuso}
If $G$ is a 2-generated profinite group, then $V(G)$ is a closed subgroup of $G.$
\end{prop}
\begin{proof}
We prove that $G\setminus V(G)$ is an open subset of $G.$ Let $x\notin V(G).$ By the previous lemma, there exists $N\in \mathcal N,$ such that $x\notin V(G/N)$. This means that
$\langle x, y \rangle N \neq G$ for every  $y\in G,$ and consequently $\langle xn, y \rangle \neq G$ for every $n\in N$ and $y\in G.$
 This implies $xN
\cap V(G)=\emptyset,$ so $xN$ is an open neighbourhood  of $x$ contained in $G\setminus V(G).$ 
\end{proof}

\begin{proof}[Proof of Theorem \ref{diquattro}]
	By \cite[Theorem 1]{diam}, for every $N$ in the set $\mathcal N$ of the open normal subgroups of $G,$ the graph $\Delta(G/N)$ is connected and $\diam(G/N)\leq 3.$ Let $a$ and $b$ be two distinct elements of $V(G)$. For every $N\in \mathcal N,$ let $d_N(a,b)$ be the distance in the graph $\Delta(G/N)$ of the two vertices $aN$ and $bN.$
	Let $$t:=
	\max_{N\in \mathcal N}d_N(a,b).$$
	Clearly $t\leq 3.$
	Set $\mathcal M=\{N\in \mathcal N \mid d_N(a,b)=t\}.$ If $N\in \mathcal N$ and $M\in \mathcal M,$ then $N\cap M\in \mathcal M,$ so $\cap_{M\in \mathcal M}M=1.$ For every $M\in \mathcal M,$ let
	$$\Omega_M\!=\!\{(x_1,\dots,x_t)\!\in\! G^t\!\mid \!\langle x_1, x_2\rangle M\!=\!\dots\!=\langle x_{t-1}, x_t\rangle M\!=\!G, x_1M\!=\!aM, x_tM\!=\!bM\}.$$ If $(x_1,\dots,x_t)\in \Omega_M,$ then $(x_1,\dots,x_t)M^t\subseteq \Omega_M$, so $\Omega_M$ is a closed subset of $G^t.$ If $M_1,\dots,M_u\in \mathcal M,$ then $\emptyset \neq \Omega_{M_1\cap \dots \cap M_u}\subseteq \Omega_{M_1}\cap \dots \cap \Omega_{M_u}.$ Since $G$ is compact, $\cap_{M\in \mathcal M}\Omega_M\neq \emptyset.$ Let $(x_1,\dots,x_t)\in \cap_{M\in \mathcal M}\Omega_M.$ Since $\langle x_1,x_2 \rangle M=\dots=\langle x_{t-1},x_t \rangle M
	=G$ for every $M\in \mathcal M,$ we have $\langle x_1,x_2 \rangle=\dots=\langle x_{t-1},x_t \rangle=G.$ Moreover $x_1\in \cap_{M\in \mathcal M}aM=\{a\}$ and $x_t\in \cap_{M\in \mathcal M}bM=\{b\}.$ We conclude that $(x_1,\dots,x_t)$ is a path in $\Delta(G)$ joining the vertices $a=x_1$ and $b=x_t.$
\end{proof}
\section{An example}\label{esempio}
Let $G_p=(\ssl(2,2^p))^{\delta_p},$ where $\delta_p$ is the largest positive integer with the property that the direct power  $(\ssl(2,2^p))^{\delta_p}$ can be generated by 2-elements. The graph $\Delta(G_p)$ is connected for every prime $p,$ and, by \cite[Theorem 1.3]{ele}, there exists an increasing sequence $(p_n)_{n\in \mathbb N}$ of odd primes, such that $\diam(\Delta(G_{p_n}))\geq 2^n$ for every $n\in \mathbb N.$ Consider the cartesian product $$G=\prod_{n\in \mathbb N}G_{p_n},$$ with the product topology. Notice that $G$ is a 2-generated profinite group. Moreover $\Delta(G)$ is the infinite tensor products of the finite graphs $\Delta(G_{p_n}),$
$n\in \mathbb N,$ and $V(G)=\prod_{n\in \mathbb N}V(G_{p_n}).$ Indeed $\langle (x_n)_{n\in \mathbb N},  (y_n)_{n\in \mathbb N}\rangle=G$ if and only if $\langle x_n, y_n\rangle=G_{p_n}$
for every $n\in \mathbb N.$ First we describe the connected components of the graph $\Delta(G).$
\begin{lemma}\label{coco}
Let $x=(x_n)_{n\in \mathbb N}\in V(G)$ and let $\Omega_x$ be the connected component of $\Delta(G)$ containing $x.$ Then $y=(y_n)_{n\in \mathbb N}$ belongs to $\Omega_x$ if and only if
$$\sup_{n\in \mathbb N}\dist_{\Delta(G_{p_n})}(x_n,y_n)<\infty.$$
\end{lemma}

\begin{proof}
Assume that $y=(y_n)_{n\in \mathbb N}\in \Omega_x$ and let $m=\dist_{\Delta(G)}(x,y).$ It follows that
$\dist_{\Delta(G_{p_n})}(x_n,y_n)\leq m$  for every $n\in \mathbb N.$
Conversely assume  $y=(y_n)_{n\in \mathbb N}\in V(G)$ with $\dist_{\Delta(G_{p_n})}(x_n,y_n)\leq m$  for every $n\in \mathbb N.$ There is a path $$x_{n,0}=x_n,x_{n,1},\dots,x_{n,\mu_n}=y_n,$$ with $\mu_n\leq m,$ joining $x_n$ and $y_n$ in the graph $\Delta(G_{p_n}).$ For $0\leq i\leq m,$ set
$$\begin{aligned}
\tilde x_{n,i}=\begin{cases}x_{n,i}&\text{if $i < \mu_n$,}\\x_{n,\mu_n}&\text{if $i\geq \mu_n$ and $m-\mu_n$ is even,}\\x_{n,\mu_n-1}x_{n,\mu_n}&\text{if $i\geq \mu_n$ and $m-\mu_n$ is odd.} 
\end{cases}
\end{aligned}$$
Then $$x= x_0=(\tilde x_{n,0})_{n\in \mathbb N},\ x_1=(\tilde x_{n,1})_{n\in \mathbb N},\dots,y=x_m=(\tilde x_{n,m})_{n\in \mathbb N}
$$ is a path joining $x$ and $y$ in the graph $\Delta(G),$ so $y\in \Omega_x.$
\end{proof}

\begin{prop}
$\Delta(G)$ has $2^\aleph_0$ different connected components.
\end{prop}
\begin{proof}
Fix $x=(x_n)_{n\in \mathbb N}\in \Delta(G).$ Let $\tau$ be
a  real number with $\tau>1.$ Since $$\diam(\Delta(G_{p_n}))\geq 2^n\geq 1+\lfloor n/\tau \rfloor,$$ for every $n\in \mathbb N$
there exists $y_{\tau,n}\in G_{p_n}$ such that
 $\dist_{\Delta(G_{p_n})}(x_n,y_{\tau,  n}))=1+\lfloor n/\tau \rfloor.$ If $\tau_2>\tau_1,$ then
$$\dist(y_{\tau_2,n},y_{\tau_1,n})\geq
\dist(x_n,y_{\tau_2,n})-\dist(x_n,y_{\tau_1,n})=\lfloor n/\tau_1 \rfloor - \lfloor n/\tau_2 \rfloor$$
tends to infinity with $n,$ so, by Lemma \ref{coco}, $\Omega_{y_{\tau_1}}\neq \Omega_{y_{\tau_2}}.$
\end{proof} 

\begin{prop}
	Let $\Omega_x$ be the connected component if $\Delta(G)$ containing $x.$ Then $\Omega_x$ is a dense subset of $V(G),$ and consequently it is not a closed subset of $V(G).$
\end{prop}
\begin{proof}
Let $y=(y_n)_{n\in \mathbb N}\in V(G):$ a base of open neighbourhoods of $y$ consists of the subsets  $\Delta_{y,m}=\{(z_n)_{n\in \mathbb N}\mid z_n=y_n \text { for every $m\leq n$}\},$
with $m\in \mathbb N.$ By Lemma \ref{coco}, $(y_1,\dots,y_m,x_{m+1},\dots,x_t,\dots)\in \Omega_x \cap \Delta_{y,m}.$
\end{proof}

\begin{lemma}\label{swa}
	If $(x_1,y_1)$ and $(x_2,y_2)$ are in the same connected component of $\Sigma_2(G),$ then $x_1, y_1, x_2, y_2$ are in the same connected component of $\Delta(G).$
\end{lemma}

\begin{proof}It suffices to prove that if $$(\alpha_1,\beta_1),\dots,(\alpha_u,\beta_u)$$
	is a path in $\Sigma_2(G)$, then the vertices $\alpha_1,\beta_1,\alpha_2,\beta_2,\dots, \alpha_u, \beta_u$ belong to the same connected component of the graph $\Delta(G).$ We prove this claim by induction on $u.$ The sentence is clearly true when $u=1.$ Assume $u\geq 2.$ By induction the vertices $\alpha_2,\beta_2,\dots, \alpha_u, \beta_u$ belong to the same connected component of  $\Delta(G);$ so it is enough to show that $\alpha_1,\beta_1,\alpha_2, \beta_2$ belong to the same connected component. Since $(\alpha_1,\beta_1)$ and $(\alpha_2,\beta_2)$
	differ for only one entry, either $\alpha_1=\alpha_2$ or $\beta_1=\beta_2.$ The graph  $\Delta(G)$ contains the path
	$\beta_1,\alpha_1=\alpha_2,\beta_2$ in the first case and the path $\alpha_1,\beta_1=\beta_2,\alpha_2$ in the second case. 
\end{proof}

\section{Degrees in the generating graph}

Before proving Theorem \ref{infinito},
we briefly recall some necessary definitions and results. Given a subset
$X$ of 
a finite group $G,$ we will denote by $d_X(G)$ the smallest cardinality
of a set of elements of $G$ generating $G$ together with the elements 
of $X.$ The following 
generalizes
a result originally
obtained by W. Gasch\"utz \cite{Ga} for $X=\emptyset.$ 

\begin{lemma}[\cite{CLis} Lemma 6] \label{modg} Let $X$ be a subset of $G$ and $N$ a normal subgroup of $G$ and suppose that
	$\langle g_1,\dots,g_k, X\rangle N=G.$ 
	%\comment{CMRD: I have swapped $k$ and $r$ in Lemma 2.11 and proof of
	%Thm 2.10 to
	%  match what we used earlier}
	If $k\geq d_X(G),$ then there exist
	$n_1,\dots,n_k\in N$ so that $\langle g_1n_1,\dots,g_kn_k,X\rangle=G.$
\end{lemma}

It follows from the proof of \cite[Lemma 6]{CLis} that the number, say  $\phi_{G,N}(X,k),$ of $k$-tuples $(g_1n_1,\dots,g_kn_k)$ generating $G$ with $X$ is 
independent of the choice of $(g_1,\dots,g_k).$ In particular 
$$\phi_{G,N}(X,k)=|N|^kP_{G,N}(X,k)$$
where $P_{G,N}(X,k)$ is the conditional probability that $k$ elements of $G$ generate $G$ with $X,$ given that they generate $G$ with $XN$. In particular, if
$$1=N_0<N_1<\dots<N_t=G$$is  a chief series of $G,$ then 
$$\phi_{G,N}(X,k)=\prod_{1\leq i \leq t}\phi_{G/N_{i-1},N_i/N_{i-1}}(XN_{i-1},k).$$
We apply the previous observations in the particular case when $X=\{g\}$ with $g\in V(G)$ and $k=1.$ In this case, setting $\delta_{G/N_{i-1},N_i/N_{i-1}}(g)=\phi_{G/N_{i-1},N_i/N_{i-1}}(gN_{i-1},1),$ we get
\begin{equation}\label{prodo}
\delta_G(g)=\prod_{1\leq i\leq t}\delta_{G/N_{i-1},N_i/N_{i-1}}(g).
\end{equation}

%\begin{prop}[\cite{xdir} Proposition 16]\label{proba} If $N$ is a normal subgroup of a finite group $G$ and $k$ is a positive integer, then	$$P_{G,N}(X,k)=\sum_{\substack{X\subseteq H\leq G\\			HN=G\\}}\frac{\mu(H,G)}{|G:H|^k}$$	where	$\mu$ is the M\"{o}bius function associated with the subgroup	lattice of $G.$ 	\end{prop}

\begin{lemma}\label{doppio}
	Let $N$ be  a minimal normal subgroup of a finite group $G$ and let $g\in V(G).$ If either $N\not\leq\frat(G)$ or $|N|>2,$ then $\delta_{G,N}(g)\geq 2.$
\end{lemma}
\begin{proof}
If $N\leq \frat(G),$ then we have $\langle g, xn \rangle =G,$ whenever $\langle g, x\rangle=G$ and $n\in N,$ so $\delta_{G,N}(g)=|N|.$ We may assume $N\not\leq \frat(G).$ We distinguish two cases:

\noindent 1) $N$ is abelian. Let $q=|\End_G(N)|$ and $r=\dim_{\End_G(N)}N.$ By \cite[Corollary 7]{diam}, $P_{G,N}(g,1)\geq \frac{q-1}{q},$ so 
$$\delta_{G,N}(g)\geq \frac{|N|(q-1)}{q}=\frac{q^r(q-1)}{q}=q^{r-1}(q-1)\geq 2,$$ except in the case $q=2$ and $r=1.$

\noindent 2) $N$ is non abelian. Choose $x$ such that $\langle g, x\rangle=G.$ Since $N$ is non abelian, by the main theorem in \cite{row}, there exists $1\neq n\in N$ such that $[g,n]=1.$ Since $N\cap Z(G)=1,$ it must be $x^n=x[x,n]\neq x.$
On the other hand, $G=\langle g, x \rangle= \langle g^n, x^n \rangle= \langle g, x[x,n]\rangle,$ so the coset $xN$ contains two different elements $x$ and $x[x,n]$ adjacent to $g$ in the graph $\Delta(G).$ This implies $\delta_{G,N}(g)\geq 2.$ 
\end{proof}

\begin{proof}[Proof of Theorem \ref{tanti}]
	Let $r$ be the number of non-Frattini factors of order 2 in a chief series of $G$. Since $C_2^r$ is an epimorphic image of $G$ and $G$ can be generated by 2 elements, it must be $r\leq 2.$ So the conclusion follows combining (\ref{prodo}) and Lemma \ref{doppio}.
\end{proof}

\begin{lemma}\label{estendo}
	Let $M$ be a closed subgroup of a 2-generated profinite group $G.$ If $g\in V(G)$ and $\langle g, x \rangle M=G,$ then there exists $m\in M$ such that $\langle g, xm\rangle=G.$
\end{lemma}
\begin{proof}
	Let $\mathcal N$ be the set of the open normal subgroups of $G$.  Given $N\in \mathcal N,$ let $\Omega_N=\{m\in M
	\mid \langle g,xm \rangle N=G\}.$ It follows from Lemma \ref{modg} that $\Omega_N\neq \emptyset.$
Moreover if $m\in \Omega_N,$ then $m(N\cap M)\subseteq \Omega_N$ and consequently $\Omega_N$ is a closed subset of $M.$ If $N_1,\dots,N_t\in \mathcal N,$ then $\emptyset \neq \Omega_{N_1\cap \dots \cap N_t}\subseteq \Omega_{N_1}\cap \dots \cap \Omega_{N_t}.$ Since $M$ is compact, $\cap_{N\in \mathcal N}\Omega_N\neq \emptyset.$ Let $m\in \cap_{N\in \mathcal N}\Omega_N:$ since $\langle g,xm \rangle N=G$ for every $N\in \mathcal N,$ we have $\langle g,xm \rangle=G.$
\end{proof}

\begin{proof}[Proof of Theorem \ref{infinito}]
	Let $g\in V(G)$ and assume, by contradiction, that $\delta_G(g)$ is finite. Set $u=\delta_G(g).$ Since $G$ is infinite, there exists an open normal subgroup $N$ of $G$ with the property that the length $t$ of a chief series of $G$ is equal to $\lceil \log_2 u\rceil+3.$ By Corollary \ref{tanti}, $\delta_{G/N}(gN)\geq 2^{t-2}=2^{\lceil \log_2 u\rceil+1}\geq 2u.$
	This means that there exist $x_1,\dots,x_{2m} \in G$ such that $x_1N\neq \dots \neq x_{2m}N$ and $\langle x_1, g\rangle N=\dots=\langle x_{2m},g\rangle=G.$
	By Lemma \ref{estendo}, there exist $n_1,\dots,n_{2m} \in N$ such that $G=\langle g, x_1n_1\rangle=\dots=\langle g, x_{2m}n_{2m}\rangle.$ This implies $u=\delta_G(g)\geq 2u,$ a contradiction.
\end{proof}

\end{document}